\documentclass{article}

\usepackage{epsfig,xspace}
\usepackage[matrix,arrow,curve]{xy}\xyoption{all}

\usepackage[usenames,dvipsnames,svgnames,table]{xcolor}
\usepackage{amssymb,amsmath,latexsym,amsfonts,amsthm,alltt}
\usepackage[dvips,usenames]{stmaryrd}
\usepackage{url}
\usepackage{graphicx}
\usepackage{float}
\usepackage{amstext}
\usepackage{cite}
\usepackage{hyperref}
\usepackage{MnSymbol}

\newtheorem{defn}{Definition}[section]
\newtheorem{rem}[defn]{Remark}
\newtheorem{thm}[defn]{Theorem}
\newtheorem{lemma}[defn]{Lemma}
\newtheorem{prop}[defn]{Proposition}
\newtheorem{coro}[defn]{Corollary}

\newtheorem{ex}{Example}[section]

\newcommand\crule[3][black]{\textcolor{#1}{\rule{#2}{#3}}}

\newcommand{\ra}{\rightarrow}
\newcommand{\lra}{\longrightarrow}

\newcommand{\la}{\leftarrow}

\newcommand{\Lra}{\Longrightarrow}

\newcommand{\midsp}{\;|\;}

\newcommand{\telos}{\hfill$\Box$}
\newcommand{\type}[1]{{\tt #1}}

\newcommand{\val}[1]{\mbox{$[\![#1]\!]$}}
\newcommand{\forces}{\Vdash}
\newcommand{\dforces}{\forces^{\!\!\partial}}
\newcommand{\yvval}[1]{\mbox{$(\!|#1|\!) $}}

\newcommand{\proves}{\vdash}

\newcommand{\vmodels}{\mbox{$\medvert\!\!\!\!\approx\;$}}
\newcommand{\zmodels}{\mbox{$\medvert\!\!\!\!\eqsim\;$}}

\newcommand{\upv}{\upVdash}
\newcommand{\rperp}{\mbox{${}^{\upv}$}}

\newcommand{\gphi}{{\mathcal  G}(Y)}
\newcommand{\gpsi}{{\mathcal  G}(X)}

\newcommand{\bbox}{\blacksquare}

\newcommand{\lperp}{{}\rperp}

\newcommand{\ldminus}{\mbox{$\;-\!\!\!\!\!\largediamond\!\!\!\!\!-\;$}}
\newcommand{\ldvert}{\mbox{$\largediamond\!\!\!\!\!\!\hspace*{-0.5pt}\mid\;$ }}

\newcommand{\ldd}{\mbox{$\largediamond\hspace*{-10pt}\Diamond\;$}}

\newcommand{\ldra}{\mbox{$\largediamond\hspace*{-10pt}\mbox{$\ra$}\;$}}
\newcommand{\ldla}{\mbox{$\largediamond\hspace*{-12.5pt}\la\;$}}

\newcommand{\lbdiamond}{\raisebox{-2pt}{\mbox{\Huge {$\filleddiamond$}}}}
\newcommand{\blackdiamond}{\raisebox{-1.5pt}{\mbox{\LARGE {$\filleddiamond$}}}}
\newcommand{\lbbox}{\raisebox{-1.2pt}[0pt][0pt]{\crule[black]{0.27cm}{0.27cm}}\hspace*{1pt}}

\newcommand{\stx}[2]{\mbox{ST$_{#1}(#2)$}}
\newcommand{\sty}[2]{\mbox{ST$_{#1}(#2)$}}

\newcommand{\lfspoon}{\leftfilledspoon}
\newcommand{\rfspoon}{\rightfilledspoon}

\title{A Characterization Result\\ for Non-Distributive Logics}
\author{Chrysafis Hartonas
\\
University of  Thessaly, Greece
\\
$\type{hartonas@uth.gr}$}

\begin{document}
\maketitle

\begin{abstract}
Recent published work has addressed the Shalqvist correspondence problem for non-distributive logics. The natural question that arises is to identify  the fragment of first-order logic that corresponds to logics without distribution, lifting van Benthem's characterization result for modal logic to this new setting. Carrying out this project is  the contribution of the present article. 

The article is intended as a demonstration and application of a project of reduction of non-distributive logics to (sorted) residuated modal logics. The reduction  is an application of recent representation results by this author for normal lattice expansions  and a generalization of a canonical and fully abstract translation of the language of  substructural logics into the language of their companion sorted, residuated modal logics. The reduction of non-distributive logics to sorted modal logics makes the proof of a van Benthem characterization of non-distributive logics nearly effortless, by adapting and reusing existing results, demonstrating the usefulness and suitability of this approach in studying logics that may lack distribution.
\end{abstract}

\section{Introduction}
Results on the model theory of non-distributive logics are quite recent. They include published results of their Shalqvist theory \cite{suzuki-shalqvist,conradie-palmigiano},  studies of a Goldblatt-Thomason theorem \cite{goldblatt-morphisms2019} for the logics of normal lattice expansions, as well as modal translation semantics \cite{pll7,redm} for non-distributive logics.

We consider here the logics of normal lattice expansions, whose relational interpretation is over sorted frames (with sorts $1,\partial$) $\mathfrak{F}=(Z_1,Z_\partial,I,\ldots)$, $I\subseteq Z_1\times Z_\partial$, and we show that the fragment of first-order formulae in the sorted first-order language of the structures $\mathfrak{F}$ equivalent to a translation of a sentence of the language of our propositional logic consists of formulae $\Phi(u)$  that are  {\em stable}, meaning that $\Phi(u)\equiv\forall^\partial v\;({\bf I}(u,v)\lra\exists^1 z\;({\bf I}(z,v)\wedge \Phi(z)))$, and invariant under sorted bisimulations.
The result is a demonstration and application of a project of reduction of non-distributive logics to (sorted) residuated modal logics. The reduction  is an application of recent representation results \cite{duality2021,sdl-exp}  for normal lattice expansions, by this author,  and a generalization of a canonical and fully abstract translation \cite{redm} of the language of  substructural logics into the language of their companion sorted, residuated modal logics.

\section{Logics of Normal Lattice Expansions}
By a {\em distribution type} we mean an element $\delta$ of the set $\{1,\partial\}^{n+1}$, for some $n\geq 0$, typically to be written as $\delta=(i_1,\ldots,i_n;i_{n+1})$ and where $\delta_{(n+1)}=i_{n+1}\in\{1,\partial\}$ will be referred to as the {\em output type} of $\delta$.  A {\em similarity type} $\tau$ is then defined as a finite sequence of distribution types, $\tau=\langle\delta_1,\ldots,\delta_k\rangle$.

An $n$-ary lattice operator $f:{\mathcal L}^n\lra {\mathcal L}$ is called {\em additive} if it distributes over finite joins of $\mathcal L$ in each argument place. If ${\mathcal L}_1,\ldots,{\mathcal L}_n, {\mathcal L}$ are bounded lattices, then a  function $f:{\mathcal L}_1\times\cdots\times{\mathcal L}_n\lra{\mathcal L}$ is {\em additive}, if for each $i$, $f$ distributes over finite joins of ${\mathcal L}_i$.

We write $\mathcal L$ for ${\mathcal L}^1$  and ${\mathcal L}^\partial$ for its opposite lattice   (where order is reversed, often designated as ${\mathcal L}^{op}$).

An $n$-ary operator $f$ on a lattice $\mathcal L$ is {\em normal} if it is an additive function $f:{\mathcal L}^{i_1}\times\cdots\times{\mathcal L}^{i_n}\lra{\mathcal L}^{i_{n+1}}$, where  each $i_j$, for  $j=1,\ldots,n,n+1$,   is in the set $\{1,\partial\}$, i.e. ${\mathcal L}^{i_j}$ is either $\mathcal L$, or ${\mathcal L}^\partial$. For a normal operator $f$ on $\mathcal L$, its {\em distribution type} is the $(n+1)$-tuple $\delta(f)=(i_1,\ldots,i_n;i_{n+1})$.

\begin{defn}[Lattice Expansions]\rm
A {\em  normal lattice expansion}  is a structure ${\mathcal L}=(L,\wedge,\vee,0,1,(f_i)_{i\in k})$ where $k>0$ is a natural number  and for each $i\in k$, $\;f_i$ is a normal operator on $\mathcal L$ of some specified arity $\alpha(f_i)\in\mathbb{N}^+$ and distribution type $\delta(i)$. The {\em similarity type} of $\mathcal L$ is the $k$-tuple $\tau({\mathcal L})=\langle\delta(0),\ldots,\delta(k-1)\rangle$.
\end{defn}

\begin{ex}\rm
A bounded lattice with a box and a diamond operator ${\mathcal L}=(L,\leq,\wedge,\vee,0,1,\boxminus,\diamondvert)$ is a normal lattice expansion of similarity type $\tau$, where  $\tau=\langle (1;1),(\partial,\partial)\rangle$ where $\delta(\diamondvert)=(1;1)$, i.e. $\diamondvert:{\mathcal L}\lra{\mathcal L}$ distributes over joins of $\mathcal L$, while $\delta(\boxminus)=(\partial;\partial)$, i.e. $\boxminus:{\mathcal L}^\partial\lra{\mathcal L}^\partial$ distributes over ``joins'' of ${\mathcal L}^\partial$ (i.e. meets of $\mathcal L$), delivering ``joins'' of ${\mathcal L}^\partial$ (i.e. meets of $\mathcal L$).

Similarly for an implicative lattice, of similarity type $\tau'=\langle(1,\partial;\partial)\rangle$ and where $(1,\partial;\partial)=\delta(\ra)$  is the distribution type of the implication operator, regarded as a map $\ra\;:{\mathcal L}\times{\mathcal L}^\partial\lra{\mathcal L}^\partial$ distributing over ``joins'' in each argument place, i.e. co-distributing over joins in the first place, turning them to meets, and distributing over meets (joins of ${\mathcal L}^\partial$) in the second place, delivering ``joins'' of ${\mathcal L}^\partial$, i.e. meets of $\mathcal L$.

An {\bf FL}-algebra (Full Lambek algebra \cite{ono3}) is a normal lattice expansion with similarity type $\tau''=\langle(1,1;1),(1,\partial;\partial),(\partial,1;\partial)\rangle$. In other words, it is a residuated lattice ${\mathcal L}=(L,\leq,\wedge,\vee,0,1,\la,\circ,\ra)$, with $\delta(\la)=(\partial,1;\partial), \delta(\circ)=(1,1;1)$ and $\delta(\ra)=(1,\partial;\partial)$.\telos
\end{ex}

Let $\mathcal{L}=(L,\leq,\wedge,\vee,0,1,\overt,\ominus)$ be a lattice expansion, where \mbox{$(L,\leq,\wedge,\vee,0,1)$} is a bounded lattice which may not be distributive, each of $\overt, \ominus$ is normal and with respective output types 1 and $\partial$. Let $\tau=\langle(i_1,\ldots,i_n;1),(i'_1,\ldots,i'_n;\partial))$ be the similarity type of $\mathcal{L}$.  The language $\Lambda_\tau$ of the lattice expansion is displayed below.
\[
\Lambda_\tau\ni\varphi:= p_i\;(i\in \mathbb{N})\midsp\top\midsp\bot\midsp\varphi\wedge\varphi\midsp \varphi\vee\varphi\midsp\overt(\overline{\varphi})\midsp\ominus(\overline{\varphi})
\]
The minimal axiomatization of the logic adds to axioms and rules for Positive Lattice Logic (the logic of bounded lattices) the normality axioms (distribution axioms) determined by the distribution types of the operators.

Normal lattice expansions are the algebraic models of non-distributive logics. Residuated (and co-residuated) lattices, in particular, have been extensively investigated, as the algebraic models of substructural logics. Consult \cite{ono-galatos} for a comprehensive presentation and literature review.

The relational (Kripke) semantics for the logics of normal lattice expansions use frames $(A,I,B,\ldots)$ where $A,B$ are sets and $I\subseteq A\times B$. Sorted relational semantic frameworks for substructural and non-distributive, more generally, logics have been proposed by Suzuki \cite{suzuki8,Suzuki-polarity-frames} and Gehrke and co-workers \cite{mai-gen,mai-grishin,dunn-gehrke}. In both cases the semantics is based on sorted representation theorems for lattices \cite{sdl,hartung}. Single-sorted approaches have been also studied \cite{vac-et-al,vak-ewa,craig}, based on a different representation \cite{urq,plo}. More recently, a representation and Stone type duality result for normal lattice expansions was presented by this author \cite{duality2021} (an improvement over the duality of \cite{sdl-exp} by the same author), extending the lattice representation of \cite{sdl}, with applications to specific cases of interest in \cite{pnsds,discres}. Structures $\mathfrak{F}=(A,\upv,B), \upv\;\subseteq A\times B$ have been introduced (named `polarities') and studied by Birkhoff \cite{birkhoff} and subsequently formed the basic structures of Formal Concept Analysis (FCA) \cite{wille2} where they are called `formal contexts'. The dual structure $\mathfrak{F}^+$ of a formal context $\mathfrak{F}$ is its `formal concept lattice', a complete lattice of `formal concepts' $(C,D)$ where $C\subseteq A$ with $C=\lperp(C\rperp)$ (a Galois stable set) and $D=C\rperp$ (a Galois co-stable set), hence also $C=\lperp D$, and where $(\;)\rperp:\powerset(A)\leftrightarrows\powerset(B):\lperp(\;)$ is the Galois connection generated by the relation $\upv$.
\begin{eqnarray*}
  C\rperp &=& \{d\in D\midsp\forall c\in C\;c\upv d\}=\{d\in D\midsp C\upv d\} \\
  \lperp D &=& \{c\in C\midsp\forall d\in D c\upv d\}=\{c\in C\midsp c\upv D\}
\end{eqnarray*}
We let $\gpsi,\gphi$ designate the complete lattices of Galois stable and co-stable sets, respectively.

Every complete lattice $\mathcal{C}$ can be represented as the formal concept lattice of the context $(\mathcal{C},\leq,\mathcal{C})$ and it was further shown in \cite{hartung} (following the FCA approach and building on Urquhart's \cite{urq}) and in \cite{sdl} (building on Goldblatt's representation of ortholattices \cite{goldb}) that every lattice can be represented as a sublattice of the formal concept lattice of a suitable formal context. This was generalized in \cite{duality2021,sdl-exp} to the case of normal lattice expansions, using sorted frames with additional relations $\mathfrak{F=}(A,I,B,(R_t)_{t\in T})$. By the similarity type of a sorted frame $\mathfrak{F}$ we shall mean the tuple $\langle\sigma(R_k)\rangle_{k\in K}$.

The relational semantics based on \cite{duality2021,sdl-exp} associates to every distribution type $\delta=(i_1,\ldots,i_n;i_{n+1})$ a sorted relation $R\subseteq Z^{i_{n+1}}\times\prod_{j=1}^{j=n}Z^{i_j}$ on the frame $(A,I,B)$, of {\em sorting type} $\sigma(R)=(i_{n+1};i_1\cdots i_n)$, where $Z^{i_j}$ is $A$, if $i_j=1$ and it is $B$ when $i_j=\partial$. Hence, to an algebra of similarity type $\tau$, a frame of the same similarity type is associated for the interpretation of the language $\Lambda_\tau$. Set operators are then canonically extracted from the relation (cf \cite{duality2021,sdl-exp,pnsds,discres} for details).  The representation is uniform and all cases reduce to the cases of relations of sorting types $(1;i_1\cdots i_n)$ and $(\partial;i'_1\cdots i'_n)$, corresponding to normal lattice operators that take their values in the lattice $\mathcal{L}$, or in its opposite lattice $\mathcal{L}^\partial$. Hence we will be only considering sorted structures $(A,I,B,R,S)$ with relations $I\subseteq A\times B$, $R\subseteq A\times\prod_{j=1}^{j=n}Z^{i_j}$ and $S\subseteq B\times\prod_{j=1}^{j=n}Z^{i'_j}$, nothing depending on having more than one relation of each sorting type, or on having relations of different arities. The Galois dual relations $R', S'$ of $R,S$  are defined by setting $R'u_1\cdots u_n=(Ru_1\cdots u_n)\rperp$ and, similarly, $S'v_1\cdots v_n=\lperp(Sv_1\cdots v_n)$. By a {\em section} of an $(n+1)$-ary relation we mean the set obtained by leaving one argument place unfilled.

\begin{prop}\rm
\label{dist from stability}
Let $\alpha_R,\alpha_S$ be the classical (but sorted) image operators generated by the relations $R,S$
\[
\alpha_R(W_1,\ldots,W_n)=\{u\midsp \exists w_1\cdots w_n(uRw_1\cdots w_n\;\wedge\;\bigwedge_j\;(w_j\in W_j))\}=\bigcup_{w_j\in W_j}^{j=1,\ldots,n}Rw_1\cdots w_n
\]
and similarly for $\alpha_S$ and let $\overline{\alpha}_R,\overline{\alpha}_S$ be the Galois closure of the restriction of $\alpha_R,\alpha_S$ to Galois stable, or co-stable sets, according to the sort type of the relations. If every section of the Galois dual relations $R',S'$ of $R,S$ is a Galois stable (or co-stable, according to the sort type) set, then $\overline{\alpha}_R,\overline{\alpha}_S$ distribute over arbitrary joins in each argument place.
\end{prop}
\begin{proof}
  Section stability was first invoked by Gehrke in \cite{mai-gen} in modeling the implication-fusion fragment of the Lambek calculus, where operators were generated by relations which are in effect the Galois dual relations $R',S'$ of $R,S$. The argument was generalized by Goldblatt \cite{goldblatt-morphisms2019} to the case of operators that either distribute in each argument place over arbitrary joins of Galois stable sets, returning a join, or they distribute over arbitrary meets of Galois stable sets, returning a meet. The argument can be generalized to that of operators of an arbitrary distribution type. The proof was given in \cite{kata2z}, to which we refer the reader for details.
\end{proof}
The operator $\overline{\alpha}_R$ (similarly for $\overline{\alpha}_S$) is  sorted  and its sorting is inherited from the sort type of $R$. For example, if $\sigma(R)=(\partial;11)$,  $\alpha_R:\powerset(X)\times\powerset(X)\lra\powerset(Y)$, hence $\overline{\alpha}_R:\gpsi\times\gpsi\lra\gphi$.
Single sorted operations
\[
\mbox{$\overline{\alpha}^1_R:\gpsi\times\gpsi\lra\gpsi$ and $\overline{\alpha}^\partial_R:\gphi\times\gphi\lra\gphi$}
\]
can be then extracted by composing appropriately with the Galois connection: $\overline{\alpha}^1_R(F,C)=(\overline{\alpha}_R(F,C))'$ (where $F,C\in\gpsi$) and, similarly, $\overline{\alpha}^\partial_R(G,D)=\overline{\alpha}_R(G',D')$ (where $G,D\in\gphi$). Similarly for the $n$-ary case and for an arbitrary distribution type.

\begin{rem}\rm\label{canonical representation remark}
A lattice operator $\phi$ of distribution type $\delta(\phi)=(i_1,\ldots,i_n;i_{n+1})$ is canonically represented in \cite{duality2021,sdl-exp} as the operator $\overline{\alpha}^1_R$, where the relation $R$, of sort type $\sigma(R)=(i_{n+1};i_1\cdots i_n)$, is defined classically by the condition
\[
uRw_1\cdots w_n\;\mbox{ iff }\forall a_1\cdots a_n\left(\bigwedge_{j=1}^{j=n}\!(a_j\in w_j)\lra \phi(a_1,\ldots,a_n)\in u\right)
\]
and it can be shown (cf. \cite{kata2z,duality2021}) that the section stability requirement of Proposition \ref{dist from stability} holds in the canonical frame construction.
\end{rem}

Frames $\mathfrak{F}=(A,I,B,R,S)$ of similarity type $\tau$ are considered for the relational semantics of the language $\Lambda_\tau$ and models $\mathfrak{M}=(\mathfrak{F},V)$ are equipped with an interpretation that assigns a Galois-stable set $V(p_i)$  to each propositional variable. A co-interpretation $V\rperp$ is also defined, by setting $V\rperp(p_i)=V(p_i)\rperp$. Interpretation and co-interpretation are extended to all sentences and we write $\val{\varphi}\subseteq A,\yvval{\varphi}\subseteq B$, respectively, so that $\yvval{\varphi}=\val{\varphi}\rperp$. Equivalently, we may say that each sentence $\varphi$ is interpreted as a formal concept $(\val{\varphi},\yvval{\varphi})$ in the formal concept lattice of the frame. The satisfaction and co-satisfaction (refutation) relations $\forces\;\subseteq A\times\Lambda,\dforces\;\subseteq B\times\Lambda$ are defined as usual, $a\forces\varphi$ iff $a\in\val{\varphi}$, $b\dforces\varphi$ iff $b\in\yvval{\varphi}$. Since interpretation and co-interpretation determine each other, it suffices to provide for each logical operator the clause for either $\forces$, or $\dforces$, as we do in Table \ref{sat}. The relations $R', S'$ in Table \ref{sat} are the Galois dual relations of $R,S$ and we recall that they are defined by setting $R'u_1\cdots u_n=(Ru_1\cdots u_n)\rperp$ and, similarly, $S'v_1\cdots v_n=\lperp(Sv_1\cdots v_n)$.

\begin{table}[!htbp]
\caption{(Co)Satisfaction relations}
\label{sat}
\begin{tabbing}
$a\forces p_i$\hskip8mm\=iff\hskip3mm\= $a\in V(p_i)$\\
$a\forces\top$ \>iff\> $a=a$\\
$b\dforces\bot$\>iff\> $b=b$\\
$a\forces\varphi\wedge\psi$\>iff\> $a\forces\varphi$ and $a\forces\psi$\\
$b\dforces\varphi\vee\psi$\>iff\> $b\dforces\varphi$ and $b\dforces\psi$\\
$b\dforces \overt(\varphi_1,\ldots,\varphi_n)$\\
\>iff\> $\forall u_1\cdots u_n\;(\bigwedge_j^{i_j=1}(u_j\forces\varphi_j)\wedge\bigwedge_r^{i_r=\partial}(u_r\dforces\varphi_r)\lra bR' u_1\cdots u_n)$\\
$a\forces\ominus(\varphi_1,\ldots,\varphi_n)$\\
\>iff\> $\forall v_1\cdots v_n\;(\bigwedge_j^{i_j=1}(v_j\forces\varphi_j)\wedge\bigwedge_r^{i_r=\partial}(v_r\dforces\varphi_r)\lra aS'v_1\cdots v_n)$
\end{tabbing}
\end{table}

Soundness of the logics of normal lattice expansions is proven in the class of frames where the relations $R,S$ satisfy the section stability requirement of Proposition \ref{dist from stability}. Completeness is shown by applying the representation arguments of \cite{duality2021,sdl-exp}, see Remark \ref{canonical representation remark}.

\begin{ex}\rm
If $\overt=\circ$ (the fusion (cotenability) operator) and $\ominus=\ra$ (implication), of respective distribution types $(1,1;1), (1,\partial;\partial)$, the semantic clauses  run as follows
\begin{tabbing}
\hskip5mm\=$b\dforces\varphi\circ\psi$\hskip3mm\=iff\hskip3mm\= $\forall a,c\in A\;(c\forces\varphi\;\wedge\;a\forces\psi\;\lra\;bR'ca)$\\
\>$a\forces\varphi\ra\psi$ \>iff\> $\forall c\in A\;\forall b\in B\;(c\forces\varphi\;\wedge\;b\dforces\psi\;\lra\;bS'ca)$
\end{tabbing}
where suitable conditions on $R,S$ ensure residuation of the operators (cf \cite{discres} for details).

For another example, consider the case of modal operators $\ominus=\boxminus,\overt=\diamondvert$, of respective distribution types $(\partial;\partial)$ and $(1;1)$. The semantic clauses run as follows, after some logical manipulation of the respective clauses in Table \ref{sat}, with $\diamondvert$ dually interpreted as necessity (cf \cite{pnsds,odigpl} for details),
\begin{tabbing}
\hskip5mm\=$b\dforces \diamondvert\varphi$ \hskip3mm\=iff\hskip3mm\= $\forall d\in B\;(bR''d\lra d\dforces\varphi)$\\
\>$a\forces \boxminus\varphi$ \>iff\> $\forall c\in A\;(aS''c\lra c\forces\varphi)$
\end{tabbing}
where we define $bR''=\lperp(bR')$ (and recall that $R'$ is defined from $R\subseteq A\times A$ by setting $R'c=(Rc)\rperp$) and similarly $S''$ is defined from $S\subseteq B\times B$ by first letting $S'b=\lperp(Sb)$, then defining $aS''=\lperp(aS')$.\telos
\end{ex}

A translation of the language $\Lambda_\tau$ into a (necessarily sorted) first-order language needs to take into account the fact that propositional variables are interpreted as Galois stable sets $C=\lperp(C\rperp)$, hence merely introducing a unary predicate {\bf P}$_i$ for each propositional variables $p_i$ falls short of the goal. The obstacle can be sidestepped by observing that the complement $I$ of the relation $\upv$ generates a pair of residuated operators $\largediamond:\powerset(A)\leftrightarrows\powerset(B):\lbbox$ such that the generated (by composition) closure operators coincide with those generated by the Galois connection, i.e. $\lperp(U\rperp)=\lbbox\largediamond U$ and $(\lperp V)\rperp=\largesquare\lbdiamond V$, for $U\subseteq  A, V\subseteq B$ and where $\largesquare=-\largediamond -$ and $\lbdiamond=-\lbbox-$. This leads to considering the sorted modal logic of polarities with relations, a project initiated in \cite{pll7,redm}.

\section{Sorted Modal Logics of Polarities with Relations}
\subsection{Sorted Residuated Modal Logic}
Fix any $\tau$-structure (a structure of similarity type $\tau$) $\mathfrak{F}=(A,B,I,R,S)$, where $I\subseteq A\times B$, and $R,S$ are $(n+1)$-ary relations of respective sorting types $\sigma(R)=(1;i_1,\ldots,i_n)$ and $\sigma(S)=(\partial;i'_1,\ldots,i'_n)$, i.e. $R\subseteq A\times\prod_{j=1}^{j=n}Z_{i_j}$ and $S\subseteq B\times\prod_{j=1}^{j=n}Z_{i'_j}$, where $Z_1=A$ and $Z_\partial=B$. The relation $I$ generates residuated operators $\largediamond:\powerset(A)\leftrightarrows\powerset(B):\lbbox$
\begin{eqnarray}
\largediamond U&=&\{b\in B\midsp\exists a\in A\;(aI b\wedge a\in U)\}\nonumber \\
\lbbox V&=&\{a\in A\midsp \forall b\in B\;(aI b\lra b\in V)\}\label{resid}
\end{eqnarray}
and each of $R,S$ generates a sorted image operator in the sense of \cite{jt1,jt2}
\begin{eqnarray}
\ldvert(U_1,\ldots,U_n) &=& \{a\in A\midsp\exists u_1\cdots u_n\;(aRu_1\cdots u_n\wedge\bigwedge_j (u_j\in U_j))\}\\
\ldminus(V_1,\ldots,V_n) &=& \{b\in B\midsp\exists v_1\cdots v_n\;(bSv_1\cdots v_n\wedge\bigwedge_j (v_j\in V_j))\}\label{image ops}
\end{eqnarray}
where for each $j=1,\ldots,n$, $U_j\subseteq Z_{i_j}\in\{A,B\}$ and $V_j\subseteq Z_{i'_j}\in\{A,B\}$.

To the structure (frame) $\mathfrak{F}=(A,B,I,R,S)$ with sorting types of $R,S$ as above, we may associate a residuated sorted modal logic with residuated modal operators $\Diamond,\bbox$ and sorted polyadic diamonds $\diamondvert,\diamondminus$ of sorting types determined by the sorting types of the relations. The language $L=(L_1,L_\partial)$ of sorted, residuated modal logic, given countable, nonempty and disjoint sets of propositional variables, is defined as follows
\begin{eqnarray*}
L_1\ni\alpha, \zeta, \eta &:=& P_i\;(i\in\mathbb{N})\midsp\neg\alpha\midsp\alpha\ra\alpha\midsp\bbox\beta\midsp\diamondvert(\overline{\theta})\\
L_\partial\ni\beta,\delta,\xi &:=& Q_i\;(i\in\mathbb{N})\midsp\neg\beta\midsp\beta\ra\beta\midsp\Box\alpha\midsp\diamondminus(\overline{\theta'})
\end{eqnarray*}
where $\overline{\theta}=(\theta_1,\ldots,\theta_n)$, the sorting type of $\diamondvert$ is $\sigma(\diamondvert)=(i_1,\ldots,i_n;1)$ and if $\sigma(j)=1$, then $\theta_j\in L_1$, else $\theta_j\in L_\partial$. Similarly for $\diamondminus$, of sorting type $\sigma'=(i'_1,\ldots,i'_n;\partial)$. Nothing of significance for our purposes is obtained by proliferating diamonds (and frame relations) and considering indexed families $(\diamondvert_j)_{j\in J}, (\diamondminus_k)_{k\in K}$ of each. Note that $\sigma(\bbox)=(\partial;1)$ and $\sigma(\Box)=(1;\partial)$. The same symbols are used for negation and implication in the two sorts and we rely on context to disambiguate. Diamond operators $\Diamond=\neg\Box\neg$ and $\blackdiamond=\neg\bbox\neg$ are defined as usual, except that the two occurrences of negation in each definition are of different sort. Sorted box operators $\boxvert, \boxminus$ are defined accordingly from $\diamondvert, \diamondminus$ and negation. Disjunction and conjunction for each sort is defined in the classical way. We let $\top,\bot\in L_1$ and $\type{t},\type{f}\in L_\partial$ designate the (definable) true and false constants for each sort. The operators ${}^\bot(\;),(\;)^\bot$ are defined by ${}^\bot\beta=\bbox(\neg\beta)$ and $\alpha^\bot=\Box(\neg\alpha)$.

\begin{rem}\rm
The sorted residuated companion modal logic of the logic of a normal lattice expansion is determined by the similarity type of the expansion. As an example, consider {\bf FL}, the associative Lambek calculus, with algebraic semantics in residuated lattices $\mathcal{L}=(L,\leq,\wedge,\vee,0,1,\la,\circ,\ra)$, of similarity type $\tau=\langle(1,1;1),(1,\partial;\partial),(\partial,1;\partial)\rangle$, where $\delta(\la)=(\partial,1;\partial)$, $\delta(\circ)=(1,1;1)$ and $\delta(\ra)=(1,\partial;\partial)$. Its companion modal logic includes three diamond operators $\ldla,\ldd,\ldra$ of respective sorting type $(\partial,1;\partial)$, $(1,1;1)$ and $(1,\partial;\partial)$. An implication operator of the first sort is defined by setting $\alpha\rfspoon\eta$  $={}^\perp(\alpha\ldra\eta^\perp)$ (reminiscent of the classical definition of implication as \mbox{$\varphi\ra\psi$}  $= \neg(\varphi\wedge\neg\psi)$) and similarly for $\lfspoon$. The two languages are interpreted over the same class of frames, determined by the similarity type at hand. Frames $\mathfrak{F}=(A,I,B,L,F,R)$ include ternary relations of respective sorting types $\sigma(L)=(\partial;\partial 1)$, $\sigma(F)=(1;11)$ and $\sigma(R)=(\partial;1\partial)$, in other words $L\subseteq B\times (B\times A)$, $F\subseteq A\times(A\times A)$, while $R\subseteq A\times(A\times B)$. Operators are generated on both arbitrary (as classical image operators) and (co)stable subsets (as closures of suitable compositions of the image operators with the Galois connection of the frame), as detailed in \cite{sdl-exp}. Appropriate frame conditions ensure that residuation obtains, as detailed in \cite{discres,redm}. In this particular case, two of the relations can be dispensed with, as they are definable in terms of the third and the Galois connection (cf \cite{discres,redm} for details).
\end{rem}

Given a sorted frame $\mathfrak{F}=(A,I,B,R,S)$ as above, a model $\mathfrak{M}=(\mathfrak{F},V)$ on the frame $\mathfrak{F}$ is equipped with a sorted valuation function $V$ such that $V(P_i)\subseteq A$ and $V(Q_i)\subseteq B$. The sorted modal language is interpreted in the expected way, as in Table \ref{sorted int}, where we use $\models\;\subseteq A\times L_1$ and $\vmodels\subseteq B\times L_\partial$ for the two satisfaction relations and, to simplify notation, we let $\zmodels$ be either $\models$ or $\vmodels$, as appropriate for the sorting type at hand. Furthermore, we let $\val{\alpha}_\mathfrak{M}\subseteq A$ and $\yvval{\beta}_\mathfrak{M}\subseteq B$ be the generated interpretations of the sorted modal formulae of the first and second sort, respectively.
{\small
\begin{table}[h]
\caption{Sorted Interpretation}
\label{sorted int}
\begin{tabbing}
$a\models P_i$\hskip6mm\= iff\hskip3mm\= $a\in V(P_i)$\hskip3cm \= $b\vmodels Q_i$\hskip6mm\= iff\hskip3mm\= $b\in V(Q_i)$\\
$a\models\neg\alpha$\>iff\> $a\not\models\alpha$\>  $b\vmodels\neg\beta$\>iff\> $b\not\vmodels\beta$\\
$a\models\alpha\wedge\eta$ \>iff\> $a\models\alpha$ and $a\models\eta$ \> $b\vmodels\beta\wedge\delta$ \>iff\> $b\vmodels\beta$ and $b\vmodels\delta$\\
$a\models\bbox\beta$ \>iff\> $\forall b\;(aIb$ implies $b\vmodels\beta)$ \>
$b\vmodels\Box\alpha$ \>iff\> $\forall a\;(aIb$ implies $a\models\alpha)$\\[2mm]
$a\models\diamondvert(\overline{\theta})$ \>iff\> $\exists w_1,\ldots,w_n\;(aRw_1\cdots w_n$ and $\bigwedge_{j=1}^n(w_j\zmodels\theta_j))$\\[2mm]
$b\vmodels\diamondminus(\overline{\theta})$ \>iff\> $\exists w_1,\ldots,w_n\;(bSw_1\cdots w_n$ and $\bigwedge_{j=1}^n(w_j\zmodels\theta_j))$
\end{tabbing}
\end{table}
}

Let $\mathbb{F}$ be a class of sorted frames and $\mathbb{M}$ the class of models $\mathfrak{M}=(\mathfrak{F},V)$ over frames $\mathfrak{F}\in\mathbb{F}$. The following notions have a standard definition just as in the single sorted case. Consult \cite{modlog} for details.
\begin{itemize}
\item $\alpha$ (or $\beta$) is {\em locally true}, or {\em satisfiable} in $\mathfrak{F}=(A,I,B,R,S)$
\item $\alpha$ (or $\beta$) is {\em globally true} in $\mathfrak{F}$
\item $\alpha$ (or $\beta$) is {\em valid in the class} $\mathbb{F}$ of frames
\item A set $\Sigma$ of sentences (perhaps of both sorts) {\em defines} the class $\mathbb{F}$ of frames.
\end{itemize}

The weakest sorted normal modal logic is an extension of a {\bf KB} type of system, where the B-axioms (one for each sort) are $\alpha\ra\bbox\!\Diamond\!\alpha$ and, for the second sort, $\beta\ra\Box\blackdiamond\beta$. In addition, it includes normality axioms for the polyadic sorted diamonds $\diamondvert,\diamondminus$. Note that in the K-axiom (one for each sort) implication on both sorts is involved, witness $\Box(\alpha\ra\eta)\ra(\Box\alpha\ra\Box\eta)$. We shall refer to it as the logic 2{\bf KB}. For frames satisfying the seriality conditions
\begin{equation}
\label{D-cond}
\forall a\in A\;\exists b\in B\;aIb\hskip1.5cm \forall b\in B\;\exists a\in A\; aIb
\end{equation}
the D-axioms $\bbox\beta\ra\blackdiamond\beta$ and $\Box\alpha\ra\Diamond\alpha$  are valid. We shall refer to the corresponding system as 2{\bf KDB}.

\subsection{Translation Semantics for Logics of Normal Lattice Expansions}
In \cite{redm}, a translation of the language  of substructural logics (in the language on the signature $\{\wedge,\vee,\top,\bot,\la,\circ,\ra\}$) was introduced, proven to be full and faithful (fully abstract). To prove a characterization theorem, generalizing the van Benthem result for modal logic, we define a family of translations, parameterized on permutations of natural numbers, given an enumeration of modal sentences of the second sort.

Let $\beta_0,\beta_1,\ldots$ be an enumeration of modal sentences of the second sort and $\pi:\omega\ra\omega$ a permutation of  natural numbers. The translation T$^\bullet_\pi$  and co-translation T$^\circ_\pi$ of sentences of the language $\Lambda_\tau$ are defined as in Table \ref{trans}.

\begin{table}[!htbp]
\caption{Modal translation and co-translation}
\label{trans}
{\small
\begin{tabbing}
T$^\bullet_\pi(p_i)$\hskip0.6cm\==\hskip1mm\= $\bbox\beta_{\pi(i)}$\hskip3.5cm \=  T$^\circ_\pi(p_i)$\hskip0.6cm\==\hskip1mm\= $\Box\blackdiamond\neg \beta_{\pi(i)}$\\
T$^\bullet_\pi(\top)$\>=\> $\top$ \> $\mbox{T}^\circ_\pi(\top)$ \>=\> $\Box\bot$\\
T$^\bullet_\pi(\bot)$ \>=\> $\bbox\type{f}$   \> $\mbox{T}^\circ_\pi(\bot)$ \>=\> $\type{t}$\\
T$^\bullet_\pi(\varphi\wedge\psi)$ \>=\> T$^\bullet_\pi(\varphi)\wedge$T$^\bullet_\pi(\psi)$ \>$\mbox{T}^\circ_\pi(\varphi\wedge\psi)$\>=\> $\Box(\blackdiamond\mbox{T}^\circ_\pi(\varphi)\vee\blackdiamond\mbox{T}^\circ_\pi(\psi))$\\
T$^\bullet_\pi(\varphi\vee\psi)$\>=\> $\bbox(\Diamond \mbox{T}^\bullet_\pi(\varphi)\vee\Diamond\mbox{T}^\bullet_\pi(\psi))$
\> $\mbox{T}^\circ_\pi(\varphi\vee\psi)$ \>=\> $\mbox{T}^\circ_\pi(\varphi)\wedge\mbox{T}^\circ_\pi(\psi)$\\[1.5mm]
T$^\bullet_\pi(\overt(\varphi_1,\ldots,\varphi_n))$ \ = $\bbox\!\Diamond\!\diamondvert(\ldots,\underbrace{\mbox{T}^\bullet_\pi(\varphi_j)}_{i_j=1},\ldots, \underbrace{\mbox{T}^\circ_\pi(\varphi_r)}_{i_r=\partial},\ldots )$\\
$\mbox{T}^\circ_\pi(\overt(\varphi_1,\ldots,\varphi_n))\;$ = $\; \Box\neg\mbox{T}^\bullet_\pi(\overt(\varphi_1,\ldots,\varphi_n))$\\[1.5mm]
T$^\bullet_\pi(\ominus(\varphi_1,\ldots,\varphi_n))$ \  = $\;\bbox\neg\mbox{T}^\circ_\pi(\ominus(\varphi_1,\ldots,\varphi_n))$\\[2.5mm]
$\mbox{T}^\circ_\pi(\ominus(\varphi_1,\ldots,\varphi_n))$ \ = $\Box\blackdiamond\!\diamondminus(\ldots,\underbrace{\mbox{T}^\bullet_\pi(\varphi_j)}_{i_j=1},\ldots, \underbrace{\mbox{T}^\circ_\pi(\varphi_r)}_{i_r=\partial},\ldots )$
\end{tabbing}
}
\end{table}

\begin{thm}\rm
\label{properties of trans}
Let $\mathfrak{F}=(X,I,Y,R,S)$ be a frame (a sorted structure) and $\mathfrak{M}=(\mathfrak{F},V)$ a model of the sorted modal language. For any enumeration $\beta_0,\beta_1,\ldots$ of modal sentences of the second sort and any permutation $\pi:\omega\lra\omega$ of the natural numbers define a model $\mathfrak{N}_\pi$ on $\mathfrak{F}$ for the language $\Lambda_\tau$ by setting $V_\pi(p_i)=\lbbox\yvval{\beta_{\pi(i)}}_\mathfrak{M}$. Then for any sentences $\varphi,\psi$ of $\Lambda_\tau$
\begin{enumerate}
\item
$\val{\varphi}_\mathfrak{N}=\val{\mbox{T}^\bullet_\pi(\varphi)}_\mathfrak{M}= \val{\lbbox\neg\mbox{T}^\circ_\pi(\varphi)}_\mathfrak{M} =\val{\lbbox\largediamond\mbox{T}^\bullet_\pi(\varphi)}_\mathfrak{M}$
\item
$\yvval{\varphi}_\mathfrak{N}=\yvval{\mbox{T}^\circ(\varphi)}_\mathfrak{M}= \yvval{\largesquare\neg\mbox{T}^\bullet(\varphi)}_\mathfrak{M} =\yvval{\largesquare\lbdiamond\mbox{T}^\circ(\varphi)}_\mathfrak{M}$
\item
$\varphi\forces\psi$ iff $\mbox{T}^\bullet_\pi(\varphi)\models\mbox{T}^\bullet_\pi(\psi)$ iff $\mbox{T}^\circ_\pi(\psi)\vmodels\mbox{T}^\circ_\pi(\varphi)$
\end{enumerate}
\end{thm}
\begin{proof}
The proof is a modification of the proof given in \cite{redm}.
In \cite{redm} we gave a translation of the language of substructural logics into the language of sorted modal logic (more precisely, into the language of 2{\bf KDB}) and proved it to be fully abstract (\cite{redm}, Theorem 4.1, Corollary 4.9). The difference with the current translation is that instead of fixing the translation of propositional variables we define a family of translations, parameterized by a permutation $\pi$ on the natural numbers. This is needed in the proof of a van Benthem type correspondence result for propositional logics without distribution (Theorem \ref{van Benthem}). The second difference is that we treat here arbitrary normal operators, rather than the operators $\la,\circ,\ra$ of the language of a substructural logic.

Claim 3) is an immediate consequence of the first two, which we prove simultaneously by structural induction. Note that, for 2), the identities $\yvval{\varphi}_\mathfrak{N}= \yvval{\Box\neg\mbox{T}^\bullet(\varphi)}_\mathfrak{M} =\yvval{\Box\blackdiamond\mbox{T}^\circ(\varphi)}_\mathfrak{M}$ are easily seen to hold for any $\varphi$, given the proof of claim 1), since
\begin{tabbing}
$\yvval{\varphi}_\mathfrak{N}$\hskip2mm\==\hskip1mm\= $\val{\varphi}_\mathfrak{N}^\upv=\largesquare(-\val{\mbox{T}^\bullet(\varphi)}_\mathfrak{M})=
\yvval{\Box\neg\mbox{T}^\bullet(\varphi)}_\mathfrak{M}$
\\[2mm]
$\yvval{\varphi}_\mathfrak{N}$\>=\> $\largesquare(-\val{\mbox{T}^\bullet(\varphi)}_\mathfrak{M})= \largesquare(-\val{\bbox\Diamond\mbox{T}^\bullet(\varphi)}_\mathfrak{M})=
\yvval{\Box\blackdiamond\Box\neg\mbox{T}^\bullet(\varphi)}_\mathfrak{M}$\\
\>=\>$\yvval{\Box\blackdiamond\mbox{T}^\circ(\varphi)}_\mathfrak{M}$
\end{tabbing}
where $\largesquare$ is the set operator interpreting the modal operator $\Box$ and similarly for the other cases, differentiating set operators from logical operators by a larger font size for the first.

For the induction proof, we separate cases.
\begin{description}
  \item[(Case $p_i$)]  $\val{p_i}_\mathfrak{N}=V_\pi(p_i)=\val{\bbox\beta_{\pi(i)}}=\val{\mbox{T}^\bullet(p_i)}$, by definitions. The other two equalities of 1) are a result of residuation and of the fact that, by residuation again, every boxed formula is stable, i.e. $\bbox\beta\equiv\bbox\!\Diamond\!\bbox\beta$. Similarly for 2), using definitions and residuation.
  \item[(Case $\top,\bot,\wedge,\vee$)] See \cite{pll7}, or \cite{redm}, Theorem 4.1.
  \item[(Case $\overt$)] To prove this case, let $\ldvert$ be the sorted image operator generated by the relation $R$
  \[
  \ldvert(\ldots,\underbrace{U_j}_{{}^{U_j\subseteq A}_{\mbox{\tiny when } i_j=1}},\ldots,\underbrace{U_r}_{{}^{U_r\subseteq B}_{\mbox{\tiny when } i_r=\partial}},\ldots)=\{a\in A\midsp\exists u_1\cdots u_n(aRu_1\cdots u_n\wedge\bigwedge_{s=1}^{s=n}u_s\in U_s)\}
  \]
let also $\widehat{\ldvert}$ be the operator on $\powerset(A)$ resulting by composition with the Galois connection and defined on $W_1,\ldots,W_n\subseteq A$ by
\[
\widehat{\ldvert}(W_1,\ldots,W_n)=\ldvert(\ldots,\underbrace{W_j}_{i_j=1},\ldots,\underbrace{\largesquare(-W_r)}_{i_r=\partial},\ldots)
\]
and let $\bigovert$ be obtained as the closure of the restriction of $\widehat{\ldvert}$ on stable subsets $C_s=\lbbox\largediamond C_s\subseteq A$, for $s=1,\ldots,n$, i.e.
\[
\mbox{$\bigovert$}(C_1,\ldots,C_n)=\lbbox\largediamond\ldvert(\ldots,\underbrace{C_j}_{i_j=1},\ldots,\underbrace{\largesquare(-C_r)}_{i_r=\partial},\ldots)
\]
A dual operator $\bigovert^{\!\partial}$ on co-stable subsets $D_s=\largesquare\lbdiamond D_s\subseteq B$ is defined by composition with the Galois connection
\[
\mbox{$\bigovert$}^{\!\partial}(D_1,\ldots,D_n)=\largesquare(-\mbox{$\bigovert$}(\lbbox(-D_1),\ldots,\lbbox(-D_n)))
\]
In particular, if $C_s=\val{\varphi_s}_\mathfrak{N}$ and $D_s=\yvval{\varphi_s}_\mathfrak{N}$ we obtain
\begin{eqnarray*}
  \mbox{$\bigovert$}(\val{\varphi_1}_\mathfrak{N},\ldots,\val{\varphi_n}_\mathfrak{N}) &=& \lbbox\largediamond\ldvert(\ldots,\underbrace{\val{\varphi_j}_\mathfrak{N}}_{i_j=1},\ldots,\underbrace{\yvval{\varphi_r}_\mathfrak{N}}_{i_r=\partial},\ldots) \\
  \mbox{$\bigovert$}^{\!\partial}(\yvval{\varphi_1}_\mathfrak{N},\ldots,\yvval{\varphi_n}_\mathfrak{N}) &=& \largesquare(-\mbox{$\bigovert$}(\val{\varphi_1}_\mathfrak{N},\ldots,\val{\varphi_n}_\mathfrak{N})\\
  &=& \largesquare\lbdiamond\largesquare(-\ldvert(\ldots,\underbrace{\val{\varphi_j}_\mathfrak{N}}_{i_j=1},\ldots,\underbrace{\yvval{\varphi_r}_\mathfrak{N}}_{i_r =\partial},\ldots))\\
  &=& \largesquare(-\ldvert(\ldots,\underbrace{\val{\varphi_j}_\mathfrak{N}}_{i_j=1},\ldots,\underbrace{\yvval{\varphi_r}_\mathfrak{N}}_{i_r=\partial},\ldots))
\end{eqnarray*}
Computing membership in the sets $\mbox{$\bigovert$}(C_1,\ldots,C_n)$ and $\mbox{$\bigovert$}^{\!\partial}(D_1,\ldots,D_n)$, see \cite{sdl-exp}, Lemma 3.6, for the particular case $C_s=\val{\varphi_s}_\mathfrak{N}$ and $D_s=\yvval{\varphi_s}_\mathfrak{N}$ we obtain the interpretation of Table \ref{sat}, i.e. $\val{\overt(\varphi_1,\ldots,\varphi_n)}_\mathfrak{N}=\mbox{$\bigovert$}(\val{\varphi_1}_\mathfrak{N},\ldots,\val{\varphi_n}_\mathfrak{N})$ and $\yvval{\overt(\varphi_1,\ldots,\varphi_n)}_\mathfrak{N}= \mbox{$\bigovert$}^{\!\partial}(\yvval{\varphi_1}_\mathfrak{N},\ldots,\yvval{\varphi_n}_\mathfrak{N})$.

Comparing with the translation, which was  defined in line with the representation results of \cite{sdl-exp} by
\[
\mbox{T$^\bullet_\pi(\overt(\varphi_1,\ldots,\varphi_n))$ \ = $\bbox\!\Diamond\!\diamondvert(\ldots,\underbrace{\mbox{T}^\bullet_\pi(\varphi_j)}_{i_j=1},\ldots, \underbrace{\mbox{T}^\circ_\pi(\varphi_r)}_{i_r=\partial},\ldots )$}
\]
and using the induction hypothesis both claims 1) and 2) follow.
\end{description}
The case for $\ominus$ is similar to the case for $\overt$.
\end{proof}

\begin{coro}\rm
\label{modal characterization}
A modal formula $\alpha\in_1L_\tau$ is equivalent to a translation T$^\bullet_\pi(\varphi)$, for some permutation $\pi:\omega\ra\omega$, of a formula $\varphi$ in the language $\Lambda_\tau$ of normal lattice expansions of similarity type $\tau$ iff there is a modal formula $\beta\in_\partial L_\tau$ such that $\alpha\equiv\bbox\beta$ iff $\alpha$ is equivalent to $\bbox\!\Diamond\!\alpha$.

Similarly for a formula $\beta\in_\partial L_\tau$ and a translation T$^\circ_\pi(\varphi)$, in which case $\beta\equiv\mbox{T}^\circ_\pi(\varphi)$ iff $\beta\equiv\Box\alpha$, for some $\alpha\in_1 L_\tau$ iff $\beta\equiv\Box\blackdiamond\beta$.
\end{coro}
\begin{proof}
The direction left-to-right follows from Theorem \ref{properties of trans}. Conversely, every formula $\bbox\beta$ is in the range of a translation T$^\bullet_\pi$ for some permutation $\pi$.
\end{proof}

\begin{rem}\rm\label{distribution in modal fragment}
Call a modal formula $\alpha$ {\em stable} if it is equivalent to $\bbox\Diamond\alpha$, and analogously for {\em co-stable}. The {\em stable fragment} (analogously for the co-stable fragment) of the sorted modal logic is the fragment of modal formulae that are stable in the above sense. The translation $T^\bullet$ maps a sentence of the non-distributive logic into the stable fragment of its companion modal logic. Analogously for the co-translation. Note that in the statement and proof of Theorem \ref{properties of trans} we did not need to place any restrictions on the frame relations $R,S$ and the proof remains valid when restricting to the class of frames where the relations $R,S$ satisfy the section stability requirement of Proposition \ref{dist from stability}.
\end{rem}

\section{First-Order Languages and Structures}
A sorted subset $S=(S_1,S_\partial)\subseteq Z=(Z_1,Z_\partial)$ is {\em finite} (more generally, of cardinallity $\kappa$) iff both $S_1,S_\partial$ are finite (resp. of cardinallity $\kappa$).  The sorted membership relation $a\in_1 Z, b\in_\partial Z$ means that $a\in Z_1=A$ and, respectively, $b\in Z_\partial=B$. For a pair $c=(a,b)$, with $a\in A,b\in B$, the statement $c\in Z$ has the obvious intended meaning. A sorted function $h:Z\lra Z'$ is a pair of functions $h_1:Z_1\lra Z'_1, h_\partial:Z_\partial\lra Z'_\partial$.
A sorted $(n+1)$-ary relation is a subset $R\subseteq Z_{i_{n+1}}\times\prod_{j=1}^{n}Z_{i_j}$, where for each $j$, $i_j\in\{1,\partial\}$. The tuple $\sigma=(i_{n+1};i_1\cdots i_n)\in\{1,\partial\}^{n+1}$ is referred to as the {\em sorting type} of $R$ and $i_{n+1}\in\{1,\partial\}$ as its {\em output type}. For an $(n+1)$-ary relation we typically use the notation $uRv_1\cdots v_n$ and sometimes, for notational transparency, $uR(v_1,\ldots,v_n)$.

Consider a structure $\mathfrak{F}=(A,B,R)$, where $Z_1=A, Z_\partial=B$ are the sort sets and $R$ is a relation of some sorting type $\sigma=(i_{n+1};i_1\cdots i_n)$. Nothing of significance for our current purposes changes if we consider expansions $(A,B,(R_s)_{s\in S})$ with a tuple of relations $R_s$, with $s$ in some index set $S$, each of some sorting type $\sigma_s$.

The {\em sorted first-order language with equality} $\mathcal{L}^1_s(V_1,V_\partial,{\bf R},=_1,=_\partial)$ of a structure $\mathfrak{F}=(A,B,R)$, for some $(n+1)$-ary sorted relation, is built on a countable sorted set $(V_1,V_\partial)$ of  individual variables $v^1_0,v^1_1,\ldots$ and $v^\partial_0,v^\partial_1,\ldots$, respectively, and an $(n+1)$-ary sorted predicate {\bf R} of some sorting type $\sigma=(i_{n+1};i_1\cdots i_n)$. Well-formed (meaning also well-sorted) formulae are built from atomic formulae $v^1_r=_1v^1_t$, $v^\partial_n=_\partial v^\partial_m$ and ${\bf R}(v^{i_{n+1}}_{r_{n+1}},v^{i_1}_{r_1},\ldots, v^{i_n}_{r_n})$ using negation, conjunction and sorted quantification $\forall^1v^1_r\Phi$, $\forall^\partial v^\partial_t\Psi$. We typically simplify notation and write $\forall^1v\Phi$, $\exists^\partial v\Phi$ etc, with an understanding and assumption of well-sortedness. We assume the usual definition of other logical operators ($\vee,\ra,\exists^1,\exists^\partial$) and of free and bound (occurrences) of a variable, as well as that of a closed formula (sentence), and we follow the usual convention about the meaning of displaying variables in a formula, as in $\Phi(v^1_0,v^\partial_1)$.

Given a sorted valuation $V$ of individual variables,  $\mathfrak{F}\models_s\Phi[V]$ is defined exactly as in the case of unsorted FOL. When $V(u^1_k)=a\in A=Z_1$, we may also display the assignment in writing $\mathfrak{F}\models_s\Phi(u^1_k)[u^1_k:=a]$ and similarly for more variables occurring free in $\Phi$. A formula $\Phi$ in $n$ free variables is also referred to as an $n$-ary {\em type}. A valuation $V$ {\em realizes} the type $\Phi$ in the structure $\mathfrak{F}$ iff $V$ satisfies $\Phi$, $\mathfrak{F}\models_s\Phi[V]$. A structure $\mathfrak{F}$ {\em realizes} $\Phi$ iff some valuation $V$ does (iff $\Phi$ is satisfiable in $\mathfrak{F}$), otherwise $\mathfrak{F}$ {\em omits} the type. Similarly for a set $\Sigma$ of $n$-ary types, which will itself, too, be referred to as an $n$-ary type.

An $\mathcal{L}^1_s$-{\em theory} $T$ is a set of $\mathcal{L}^1_s$-sentences and a {\em complete theory} is a theory whose set of consequences $\{\Phi\midsp T\models_s\Phi\}$ is maximal consistent. The (complete) $\mathcal{L}^1_s$-theory of a structure is designated by $\type{Th}_s(\mathfrak{F})$.
If $C=(C_1,C_\partial)\subseteq(A,B)$ is a sorted subset, the expansion $\mathcal{L}^1_{s}[C]$ of the language includes sorted constants $c^1\in C_1, c^\partial\in C_\partial$, for each member of $C_1,C_\partial$. We sometimes simplify notation writing $c_a, c_b$ for the constants naming the elements $a\in A, b\in B$. It is assumed, as usual, that a constant is interpreted as the element that it names. The extended structure interpreting the expanded signature of the language is designated by $(\mathfrak{F},c)_{c\in C}$, or just $\mathfrak{F}_{C}$.

For a similarity type $\tau$,  $\mathcal{L}^1_{s,\tau}$ is the sorted first-order language that includes a predicate of sorting type $\sigma$, for each $\sigma$ in $\tau$, together with a distinguished binary predicate {\bf I} of sorting type $(1;\partial)$.

To a structure $\mathfrak{F}=(A,B,R)$ we may also associate an {\em unsorted (single-sorted) first-order language with equality} $\mathcal{L}^1(V',{\bf U}_1,{\bf U}_\partial,{\bf R},=)$ where the interpretation of ${\bf U}_1,{\bf U}_\partial$ is, respectively, $Z_1=A, Z_\partial=B$ and $V'=V_1\cup V_\partial$. Assuming the sorting type of $R$ is $\sigma=(i_{n+1};i_1\cdots i_n)$, the structure validates all sentences pertaining to sorting constraints, which are of the following form, with  $i_{j_r}\in\{1,\partial\}$, for each $r$.
\begin{eqnarray}
 \forall v_1\cdots\forall v_{n+1}\;({\bf R}(v_{n+1},v_1,\ldots,v_n)\lra\bigwedge_{r=1}^{r=n+1} {\bf U}_{i_{j_r}}(v_r)) \label{validate1}\\
\forall v_1\forall v_2(v_1=v_2\lra (({\bf U}_1(v_1)\wedge {\bf U}_1(v_2))\vee ({\bf U}_\partial(v_1)\wedge {\bf U}_\partial(v_2)))\label{validate2}
\end{eqnarray}
In particular, \eqref{validate2} implies the sentence $\forall v\;({\bf U}_1(v)\vee{\bf U}_\partial(v))$.
The (unsorted) $\mathcal{L}^1$-theory of $\mathfrak{F}$ will be designated by $\type{Th}(\mathfrak{F})$.

By sort-reduction (for details  cf. \cite{enderton}, ch. 4), the language $\mathcal{L}^1_s$ can be translated into $\mathcal{L}^1$, by relativising quantifiers (where $i_r\in\{1,\partial\}$)
\[
\Psi=\forall^{i_r}u^{i_r}_k\Phi\;\;\Mapsto\;\;\Psi^*=\forall u^{i_r}_k\;({\bf U}_{i_r}(u^{i_r}_k)\lra\Phi^*)
\]
and replacing $=_1, =_\partial$ by a single equality predicate $=$. For later use we list the following result.
\begin{thm}[Enderton \cite{enderton}, ch. 4.3]\rm
\label{sorted compactness}\mbox{}
\begin{enumerate}
\item (Sort-reduction) If $\Phi^*$ is the sort-reduct of $\Phi$ and $V$ a valuation of variables, then  $\mathfrak{F}\models_s\Phi[V]$ iff $\mathfrak{F}\models\Phi^*[V]$.
\item (Compactness) If every finite subset of a set $\Sigma$ of many-sorted sentences in $\mathcal{L}^1_s$ has a model, then $\Sigma$ has a model.\telos
\end{enumerate}
\end{thm}

\subsection{Standard Translation of Sorted Modal Logic}
The standard translation of sorted modal logic into sorted  FOL  is exactly as in the single-sorted case, except for the relativization to two sorts, displayed in Table \ref{std-kdb}, where $\stx{u}{}, \sty{v}{}$ are defined by mutual recursion and $u,v$ are individual variables of sort $1,\partial$, respectively.

\begin{table}[!htbp]
\caption{Standard Translation of the sorted modal language}
\label{std-kdb}
\begin{tabbing}
$\stx{u}{P_i}$\hskip0.8cm\==\hskip2mm\= ${\bf P}_i(u)$ \\
$\stx{u}{\neg \alpha}$\>=\> $\neg\stx{u}{\alpha}$ \\
$\stx{u}{\alpha\wedge \alpha'}$\>=\> $\stx{u}{\alpha}\wedge\stx{u}{\alpha'}$ \\
$\stx{u}{\bbox \beta}$ \>=\> $\forall v\;({\bf I}(u,v)\;\lra\;\sty{v}{\beta})$ \\
$\stx{u}{\diamondvert(\overline{\theta})}$ \>=\> $\exists \overline{u}\;({\bf R}(u,\overline{u})\;\wedge\;\bigwedge_{j=1,\dots,n}^{i_j=1}\stx{u_j}{\theta_j}\;\wedge\; \bigwedge_{r=1,\dots,n}^{i_r=\partial}\stx{u_r}{\theta_r})$\\[3mm]
$\sty{v}{Q_i}$\>=\> ${\bf Q}_i(v)$\\
$\sty{v}{\neg \beta}$ \>=\> $\neg\sty{v}{\beta}$\\
$\sty{v}{\beta\wedge \beta'}$\>=\> $\sty{v}{\beta}\wedge\sty{v}{\beta'}$\\
$\sty{v}{\Box \alpha}$ \>=\> $\forall u\;({\bf I}(u,v)\;\lra\;\stx{u}{\alpha})$\\
$\sty{v}{\diamondminus(\overline{\theta'})}$ \>=\> $\exists \overline{v}\;({\bf S}(v,\overline{v})\;\wedge\;\bigwedge_{j=1,\dots,m}^{i_j=1}\stx{v_j}{\theta_j}\;\wedge\; \bigwedge_{r=1,\dots,m}^{i_r=\partial}\stx{v_r}{\theta_r})$
\end{tabbing}
\end{table}

\begin{prop}\rm
\label{std trans}
For any sorted modal formulae $\alpha,\beta$ (of sort $1, \partial$, respectively), for any model $\mathfrak{M}=((A,I,B,R,S),V)$ and for any $a\in A, b\in B$, $\mathfrak{M},a\models \alpha$ iff $\mathfrak{M}\models\stx{u}{\alpha}[u:=a]$ and $\mathfrak{M},b\vmodels \beta$ iff $\mathfrak{M}\models\sty{v}{\beta}[v:=b]$.
\end{prop}
\begin{proof}
  Straightforward.
\end{proof}

We next review and adapt to the sorted case the basics on ultraproducts and ultrapowers that will be needed in the sequel. Consult \cite{chang-keisler,bell-slomson} for details.

\subsection{Sorted Ultraproducts}
Let $(\mathfrak{F}_j)_{j\in J}=(A_j,B_j,R_j)_{j\in J}$, with $J$ some index set, be a family of structures with sorted relations $R_j$ of some fixed sorting type $\sigma$.

An {\em ultrafilter over $J$} is an ultrafilter (maximal filter) $U$ of the powerset Boolean algebra $\powerset(J)$.
Let $\prod_U A_j,\prod_U B_j$ be the ultraproducts of the families of sets $(A_j)_{j\in J}, (B_j)_{j\in J}$ over the ultrafilter $U$. Members of $\prod_U A_j$ are equivalence classes $f_U$ of functions $f\in\prod_{j\in J}A_j$ (i.e. functions $f:J\lra\bigcup_jX_j$ such that for all $j\in J, f(j)\in A_j$) under the equivalence relation $f\sim_U g$ iff  $\{j\in J\midsp f(j)=g(j)\}\in U$.

Goldblatt \cite{goldblatt-definable2018,goldblatt-morphisms2019,goldblatt-ultra-2018} introduced ultraproducts for polarities (sorted structures with a binary relation), slightly generalizing the classical construction. We review the definition, adapting to the case of an arbitrary $(n+1)$-ary relation.

\begin{defn}[Ultraproducts of Sorted Structures]\rm
\label{defn ultraproducts}
Given a family $(\mathfrak{F}_j)_{j\in J}$ of structures (models) with $J$ some index set, their {\em ultraproduct} is the sorted structure $\prod_U\mathfrak{F}_j=(\prod_U A_j,\prod_U B_j,R_U)$ where
\begin{enumerate}
  \item $\prod_U A_j,\prod_U B_j$ are the ultraproducts over $U$ of the families of sets $(A_j)_{j\in J}$, $(B_j)_{j\in J}$ .
  \item
  Where the sorting type of $R_j$  for all $j\in J$ is $\sigma=(i_{n+1};i_1,\ldots,i_n)$  and for each $r\in\{1,\ldots,n,n+1\}$ we have  $h_{r,U}\in\prod_U A_j$, if $i_r=1$,  and $h_{r,U}\in\prod_U B_j$ if $i_r=\partial$, the relation $R_U$, of sorting type $\sigma$ is defined by setting
  \begin{equation}
  \label{ultraproduct relation}
  h_{n+1,U}R_U(h_{1,U},\ldots, h_{n,U})\mbox{ iff }\{j\in J\midsp h_{n+1}(j)R_j (h_{1}(j),\ldots, h_{n}(j))\}\in U
  \end{equation}
\end{enumerate}
If for all $j\in J$, $\mathfrak{F}_j=\mathfrak{F}$, then the ultraproduct is referred to as the {\em ultrapower} $\prod_U\mathfrak{F}$ of $\mathfrak{F}$ over the ultrafilter $U$.\telos
\end{defn}

Considering the structures $\mathfrak{F}_j$ as $\mathfrak{L}^1$-structures,
by the fundamental theorem of ultraproducts ({\L}os's theorem) we have
{\small
\begin{eqnarray}\label{ultraproduct semantics}
  \mbox{$\prod_U$}\mathfrak{F}_j\models\Phi[f_{1,U},\ldots,f_{t,U}]\;\mbox{ iff }
   \{j\in J\midsp \mathfrak{F}_j\models \Phi[f_{1}(j),\ldots,f_{t}(j)]\}\in U
\end{eqnarray}}
\noindent
By sort reduction, {\L}os's theorem holds when the $\mathfrak{F}_j$ are regarded as models of the sorted language (as $\mathcal{L}^1_s$-structures), as well. Indeed
\begin{tabbing}
$\mbox{$\prod_U$}\mathfrak{F}_j\models_s\Phi[f_{1,U},\ldots,f_{t,U}]$ \= iff \= $\mbox{$\prod_U$}\mathfrak{F}_j\models\Phi^*[f_{1,U},\ldots,f_{t,U}]$\\
\>iff\> $\{j\in J\midsp \mathfrak{F}_j\models \Phi^*[f_{1}(j),\ldots,f_{t}(j)]\}\in U$\\
\>iff\> $\{j\in J\midsp \mathfrak{F}_j\models_s \Phi[f_{1}(j),\ldots,f_{t}(j)]\}\in U$
\end{tabbing}
We use the standard notation $\mathfrak{F}\equiv\mathfrak{G}$ for {\em elementarily equivalent structures} (satisfying the same set of sentences) and $\mathfrak{F}\prec\mathfrak{G}$ to designate the fact that $\mathfrak{G}$ is an {\em elementary extension of} $\mathfrak{F}$, meaning that $\mathfrak{F}\subset\mathfrak{G}$ ($\mathfrak{F}$ is a substructure of $\mathfrak{G}$) and for any $n$-ary type $\Phi(w_{i_1},\ldots,w_{i_n})$ of some sort $(i_1,\ldots,i_n)\in\{1,\partial\}^n$ and any valuation $V$ for $\mathfrak{F}$ we have $\mathfrak{F}\models\Phi[V]$ iff $\mathfrak{G}\models\Phi[V]$. Finally, we recall that a map $h:\mathfrak{F}\preceq\mathfrak{G}$ is an {\em elementary embedding} iff for any $n$-ary type $\Phi$ as above we have $\mathfrak{F}\models\Phi(\overline{w_{i_j}})[V]$ iff $\mathfrak{G}\models\Phi(\overline{w_{i_j}})[h\circ V]$.

The same argument as above, appealing to sort-reduction, applies to lift to the sorted case well-known consequences of {\L}os's theorem (in particular, Corollary 4.1.13 of \cite{chang-keisler}, restated for the sorted case below).

\begin{coro}\rm
\label{elementary to ultrapower}
If $\mathfrak{F}$ is an $\mathcal{L}^1_s$-structure, $J$ an index set and $U$ an ultrafilter over $J$, then $\mathfrak{F}$ and the ultrapower $\prod_U\mathfrak{F}$ are elementarily equivalent, $\mathfrak{F}\equiv\prod_U\mathfrak{F}$.
Furthermore, the  embedding \mbox{$e=(e_1:Z_1\ra\prod_U Z_1,e_\partial:Z_\partial\ra\prod_U Z_\partial)$} sending elements $a\in Z_1=A$, $b\in Z_\partial=B$ to the respective equivalence classes $e(a)=e_1(a)=f_{a,U}$, $e(b)=e_\partial(b)=f_{b,U}$ of the constant functions  $f_a(j)=a$, $f_b(j)=b$, for all $j\in J$, is an elementary embedding $e:\mathfrak{F}\prec\prod_U\mathfrak{F}$.
\end{coro}
\begin{proof}[Sketch of Proof]
By appealing to sort-reduction (cf Theorem \ref{sorted compactness}). In fact,
   a direct argument for the sorted case is literally the same as in the unsorted case, as seen by consulting for example the proof in \cite{bell-slomson}, Lemma 2.3.
\end{proof}
Note, in particular, that for a unary type $\Phi(u^1)\in\mathcal{L}^1_s$ and any element say $a\in A$ (i.e. a valuation $V$ such that $V(u^1)=a\in A$) we have (dropping the sorting superscript on the variable) the following
\begin{coro}\rm
\label{useful coro}
For $u\in V_1$,
$\mathfrak{F}\models_s\Phi(u)[u:=a]$ \ \ iff\ \
  $\prod_U\mathfrak{F}\models_s\Phi(u)[u:=f_{a,U}]$. The same holds for a type with a free variable $v\in V_\partial$.
\end{coro}
\begin{proof}\mbox{}
\begin{tabbing}
$\prod_U\mathfrak{F}\models_s\Phi(u)[u:=f_{a,U}]$\hskip3mm\= iff \hskip3mm\= $\prod_U\mathfrak{F}\models\Phi^*(u)[u:=f_{a,U}]$ \hskip3mm\= (by sort-reduction)\\
\hspace*{1cm}\= iff \hskip3mm\= $\{j\midsp \mathfrak{F}\models\Phi^*(u)[u:=f_a(j)]\}\in U$ \> (by {\L}os's theorem)\\
\>iff\> $\{j\midsp \mathfrak{F}\models\Phi^*(u)[u:=a]\}\in U$ \> ($\forall j\;f_a(j)=a$)\\
\>iff\> $\mathfrak{F}\models\Phi^*(u)[u:=a]$ \hskip5mm ($U$ is a filter, so $\{j\midsp \mathfrak{F}\models\Phi^*(u)[u:=a]\}\neq\emptyset$)\\
\>iff\> $\mathfrak{F}\models_s\Phi(u)[u:=a]$ \hskip5mm (by sort-reduction)
\end{tabbing}
and this proves the claim.
\end{proof}

\begin{defn}[Ultrapowers of Models]\rm
If $\mathfrak{M}=(\mathfrak{F},V)$ is a model and $U$ is an ultrafilter over an index set $J$, the ultrapower of $\mathfrak{M}$ is defined by $\prod_U\mathfrak{M}=(\prod_U\mathfrak{F}, V_U)$ where $V_U(u)=f_{a,U}$ iff $V(u)=a$.
\end{defn}

\subsection{Saturated Structures}
Let $\mathfrak{F}=\langle A,B, R\rangle$ be an $\mathcal{L}^1_s$-structure. The structure $\mathfrak{F}$ is called {\em $\omega$-saturated} iff  for any  finite subset $C\subseteq A\cup B$, every unary type $\Sigma$ of the expanded language $\mathcal{L}^1_s[C]$  that is consistent with the theory $\type{Th}_s(\mathfrak{F}, c)_{c\in C}$ is realized  in $(\mathfrak{F}, c)_{c\in C}$. If reference to sorting is disregarded, this is precisely the meaning of $\omega$-saturated structures for (unsorted) first-order languages. The definition generalizes to $\kappa$-saturated structures, for any cardinal $\kappa$, but we shall only have use of $\omega$-saturated structures in the sequel.

$\omega$-saturated first-order (unsorted) structures can be constructed as unions of elementary chains, or as ultrapowers. Consult Bell and Slomson \cite{bell-slomson}, Theorem 1.7 and Theorem 2.1, or  Chang and Keisler \cite{chang-keisler}, ch. 5, for details. Any two elementarily equivalent $\kappa$-saturated structures are isomorphic (\cite{chang-keisler}, Theorem 5.1.13, \cite{bell-slomson}, Theorem 3.1), so we only discuss  ultrapowers. With some necessary adaptation, the original arguments for the unsorted case (for the existence of $\omega$-saturated extensons) can be reproduced for the sorted case. It is easier, however, to derive the result for the sorted case by reducing the problem to the unsorted case, using sort-reduction, as we do below.

\begin{thm}\rm
\label{omega-saturated thm}
Every $\mathcal{L}^1_s$-structure $\mathfrak{F}$ has an elementary $\omega$-saturated extension  $h:\mathfrak{F}\preceq\mathfrak{F}^\flat$.
\end{thm}
\begin{proof}
By standard model-theoretic results (\cite{chang-keisler}, Proposition 5.1.1, Theorem 6.1.1),  for every first-order structure ($\mathcal{L}^1$-structure) $\mathfrak{F}$ and any countably incomplete ultrafilter\footnote{As a countably incomplete ultrafilter we may take an ultrafilter over the set of natural numbers that does not contain any singletons (cf \cite{modlog}, Example 2.72).} $U$ over some index set $J$, its ultrapower $\prod_U\mathfrak{F}$ is an elementary $\omega$-saturated extension of $\mathfrak{F}$,  $e:\mathfrak{F}\prec\prod_U\mathfrak{F}$, by the embedding of (the unsorted version of) Corollary \ref{elementary to ultrapower} (see \cite{chang-keisler}, Corollary 4.1.13).

Let $\mathfrak{F}=(A,B,R)$ be a sorted first-order structure, $C\subseteq A\cup B$ and $\Sigma(v)$, with $v\in V_1\cup V_\partial$, a unary type in the expanded language $\mathcal{L}^1_s[C]$ consistent with $\type{Th}_s(\mathfrak{F}_C)$.  We claim that the sort reduct $\Sigma^*=\{\Phi^*(v)\midsp\Phi(v)\in\Sigma\}$ is consistent with the (unsorted) theory $\type{Th}(\mathfrak{F}_C)$. Assuming for the moment that the claim is proved, by $\omega$-saturation of the ultrapower of the $\mathcal{L}^1$-structure $\mathfrak{F}$, $\prod_U\mathfrak{F}_C\models\Sigma^*[S]$ and then by sort reduction $\prod_U\mathfrak{F}_C\models_s\Sigma[S]$, i.e. the type $\Sigma$ in the sorted language $\mathcal{L}^1_s[C]$ is realized in  $\prod_U\mathfrak{F}_C$ by some  valuation $S$. Hence $\prod_U\mathfrak{F}$, regarded as an $\mathcal{L}^1_s$-structure, is an elementary $\omega$-saturated extension of $\mathfrak{F}$.

To prove the claim we made in course of the above argument, recall that we assume that $\Sigma(v)$ is consistent with $\type{Th}_s(\mathfrak{F}_C)$, so that a structure $\mathfrak{N}$ and a valuation $V_N$ exist such that $V_N$ satisfies in $\mathfrak{N}$ every formula in $\Sigma(v)$ and sentence in $\type{Th}_s(\mathfrak{F}_C)$.

If $\Sigma^*$ is not consistent with the theory $\type{Th}(\mathfrak{F}_C)$, let $\Phi\in\mathcal{L}^1[C]$ be such that both $\Phi$ and $\neg\Phi$ are derivable from $\Sigma(v)\cup\type{Th}(\mathfrak{F}_C)$. By compactness, let $\Phi_1^*(v),\ldots,\Phi_n^*(v)\in\Sigma^*$ and $\Theta_1,\ldots,\Theta_k$ be sentences in $\type{Th}(\mathfrak{F}_C)$ such that $\Phi_1^*(v),\ldots,\Phi_n^*(v),\Theta_1,\ldots,\Theta_k\proves\Phi\wedge\neg\Phi$. Since $\type{Th}(\mathfrak{F}_C)$ is a complete theory we may assume that $\Phi\in\type{Th}(\mathfrak{F}_C)$, hence  $\Phi_1^*(v),\ldots,\Phi_n^*(v),\Theta_1,\ldots,\Theta_k\proves\neg\Phi$ and then  $\Phi,\Phi^*_2(v),\ldots,\Phi_n^*(v),\Theta_1,\ldots,\Theta_k\proves\neg\Phi_1^*(v)$. Since $\mathfrak{N}$ with the valuation $V_N$ satisfy each of the formulas on the left it follows $\mathfrak{N}\models\neg\Phi_1^*(v)[V_N]$. By sort-reduction (\cite{enderton}, Lemma 4.3A), $\mathfrak{N}\models_s\neg\Phi_1(v)[V_N]$, contradiction.
\end{proof}

\section{Bisimulations and van Benthem Characterization}

\subsection{Bisimulations}
\begin{defn}[Bisimulation on Sorted Structures]\rm
\label{defn bisimulation}
Let $\mathfrak{F}=(A,I,B,R,S)$, $\mathfrak{F'}=(A',I',B',R',S')$ be frames, where recall that the sorting types of $R,S$ are $(1;i_1,\ldots,i_n)$ and $(\partial;i'_1,\ldots,i'_m)$, respectively, and assume that $\precsim\;\subseteq Z\times Z'$ (where $Z=A\cup B$ and $Z'=A'\cup B'$) is a well-sorted relation (i.e. for $a\in A$, $b\in B$, the set $a\precsim\;=\{b\midsp a\precsim b\}$ is a subset of $A'$ and similarly $b\precsim\;\subseteq B'$). Then the relation $\precsim$ is a {\em simulation} iff
\begin{enumerate}
\item If $a\precsim a'$ then
    \begin{itemize}
    \item[-] if $aIb$, then $b\precsim b'$ for some $b'\in B'$ such that $a'I'b'$
    \item[-] if $aRu_1\cdots u_n$, then $a'R'u'_1\cdots u'_n$ for some $u'_j$ such that $u_j\precsim u'_j$
    \end{itemize}
\item If $b\precsim b'$ then
    \begin{itemize}
    \item[-] if $aIb$, then $a\precsim a'$ for some $a'\in A'$ such that $a'I'b'$
    \item[-] if $bSv_1\cdots v_m$, then $b'S'v'_1\cdots v'_m$ for some $v'_j$ such that $v_j\precsim v'_j$
    \end{itemize}
\end{enumerate}
A relation $\precsim$ is a {\em bisimulation} if both $\precsim$ and its inverse $\precsim^{-1}$ are simulations. We use the notation $\sim$ for bisimulations. If $\sim$ is a bisimulation for the frames, we write $\mathfrak{F}\sim\mathfrak{F'}$.
\telos
\end{defn}

A  relation $\sim$ is a bisimulation of models $\mathfrak{M}=(\mathfrak{F},V),\mathfrak{M'}=(\mathfrak{F'},V')$ iff
\begin{enumerate}
  \item $\mathfrak{F}\sim\mathfrak{F'}$ and
  \item for any propositional variable $P_i$ of the first sort, if $a\in V(P_i)$ and $a\sim a'$, then $a'\in V'(P_i)$
  \item for any propositional variable $Q_i$ of the second sort, if $b\in V(Q_i)$ and $b\sim b'$, then $b'\in V'(Q_i)$
\end{enumerate}
If $\sim$ is a bisimulation for the models, we write $\mathfrak{M}\sim\mathfrak{M'}$ and we use $\mathfrak{M},w\sim\mathfrak{M'},w'$ when $w\sim w'$ are points of either (but the same) sort.

\begin{prop}\rm
\label{modal=>invariant}
Sorted modal formulas  are invariant under bisimulation. In other words, if $\mathfrak{M},a\sim\mathfrak{M'},a'$ (resp. $\mathfrak{M},b\sim\mathfrak{M'},b'$), then $\mathfrak{M}\models\alpha(u)[u:=a]$ iff $\mathfrak{M'}\models\alpha(u)[u:=a']$ (resp. $\mathfrak{M}\vmodels\beta(v)[v:=b]$ iff $\mathfrak{M'}\vmodels\beta(v)[v:=b']$).
\end{prop}
\begin{proof}
  By structural induction, observing that the argument for the base case is built into the definition of bisimulations, while for negations and implications it reduces to that for the subsentences and the induction hypothesis is used, while for modal operators the corresponding clauses in the definition of bisimulations allow directly the use of the inductive hypothesis on subsentences.
\end{proof}

For models $\mathfrak{M,M'}$, points $a\in A, a'\in A'$, respectively are {\em modally equivalent} iff for any $\alpha\in_1 L_\tau$, $\mathfrak{M},a\models_s\alpha$ iff $\mathfrak{M'},a'\models_s\alpha$. In symbols $\mathfrak{M},a\stackrel{\Box}{\leftrightsquigarrow}\mathfrak{M'},a'$. Similarly for points $b\in B, b'\in B'$.

A kind of converse of Proposition \ref{modal=>invariant} is provided below.
\begin{prop}\rm
\label{modal equivalence => bisimulation}
Assume $\mathfrak{M},a\stackrel{\Box}{\leftrightsquigarrow}\mathfrak{M'},a'$ and let $U$ be a countably incomplete ultrafilter over some index set $J$. Then $\prod_D\mathfrak{M},f_{a,U}\stackrel{\Box}{\leftrightsquigarrow}\prod_D\mathfrak{M'},f_{a',U}$ and the relation of modal equivalence on the ultrapowers is  a bisimulation.
\end{prop}
\begin{proof}
By Proposition \ref{std trans} we have $\mathfrak{M},a\stackrel{\Box}{\leftrightsquigarrow}\prod_D\mathfrak{M},f_{a,D}$ and so the first hypothesis implies that $\prod_D\mathfrak{M},f_{a,U}\stackrel{\Box}{\leftrightsquigarrow}\prod_D\mathfrak{M'},f_{a',U}$.

From the second hypothesis and Theorem \ref{omega-saturated thm} it is obtained that $\prod_D\mathfrak{M}$, $\prod_D\mathfrak{M'}$ are $\omega$-saturated and the claim is that this implies that modal equivalence is a bisimulation. The proof of this claim is the same as the corresponding proof in the unsorted case (cf. \cite{modlog}, ch. 2, Proposition 2.54 and Theorem 2.65).
\end{proof}

If $\Phi=\Phi(u)\in\mathcal{L}$ has only the displayed variable $u\in V_1$ free (i.e. it is a unary type) and it holds that $\Phi\models \stx{u}{\alpha}$, for some $\alpha$ (of sort 1) in the sorted modal language, then we say that $\stx{u}{\alpha}$ is a {\em modal 1-consequence} of $\Phi$. Similarly, if $\Psi(v)\models\sty{v}{\beta}$ then $\sty{v}{\beta}$ is a {\em modal $\partial$-consequence} of $\Psi$. Let $m^1_1(\Phi)$, $m^1_\partial(\Psi)$ be the sets of 1- and $\partial$-consequences of $\Phi,\Psi$, respectively.

\begin{lemma}\rm
\label{inv bisim lemma}
Let $\Phi(u)$ be a sorted first-order formula in one free variable $u\in V_1$ and let $m^1_1(\Phi)$ be the set of its modal 1-consequences. If $\Phi$ is invariant under bisimulation, then $m^1_1(\Phi)\models_s\Phi$. Similarly, for a bisimulation invariant formula $\Psi(v)\in\mathcal{L}$, with $v\in V_\partial$, $m^1_\partial(\Psi)\models_s\Psi$.
\end{lemma}
\begin{proof}
The proof is again similar to that for the unsorted case, see for example the proof in Theorem 2.68 of \cite{modlog}. We provide some details.

Let $\mathfrak{M}=((A,I,B,R,S),V)$, $a\in A$, assume \mbox{$\mathfrak{M}\!\models_s\! m^1_1(\Phi)[u:=a]$}  and observe that $\Phi\cup m^1_1(a)$ is consistent. Otherwise, by compactness of sorted FOL \cite{enderton}, we obtain that $\models_s\Phi\ra\neg\bigwedge m^1_0(a)$, for some finite $m^1_0(a)\subseteq m^1_1(a)$. Hence, $\neg\bigwedge m^1_0(a)\in m^1_1(\Phi)$ which implies that $\mathfrak{M}\models_s\neg\bigwedge m^1_0(a)$. This is in contradiction with the fact that $m^1_0(a)\subseteq m^1_1(a)$ and $\mathfrak{M}\models_s \stx{u}{\alpha}[u:=a]$ for all $\stx{u}{\alpha}\in m^1_1(a)$.

By consistency of  $\Phi(u)\cup m^1_1(a)$, let  $\mathfrak{M}'=((A',I',B',R',S'),V')$ be a model and $a'\in A'$ such that $\mathfrak{M}'\models_s \{\Phi(u)\}\cup m^1_1(a)[u:=a']$. Then for any sentence $\alpha$ of the first sort in the sorted modal language, $\mathfrak{M},a\models_s\alpha$ iff $\mathfrak{M}',a'\models_s\alpha$. i.e. $a,a'$ are modally equivalent. This is because if $\mathfrak{M},a\models_s\alpha$, then $\stx{u}{\alpha}\in m^1_1(a)$ and therefore by $\mathfrak{M}'\models_s \{\Phi(u)\}\cup m^1_1(a)[u:=a']$ and Proposition \ref{std trans} it follows that $\mathfrak{M}',a'\models_s\alpha$. Conversely, if $\mathfrak{M}',a'\models_s\alpha$, then it must be that $\mathfrak{M},a\models_s\alpha$ for, if not, then $\mathfrak{M},a\models_s\neg\alpha$ and this implies $\mathfrak{M}',a'\models_s\neg\alpha$ which is a contradiction.

To obtain $\mathfrak{M}\models_s\Phi(u)[u:=a]$ from $\mathfrak{M}'\models_s\Phi(u)[u:=a']$, let $U$ be a countably incomplete ultrafilter over some index set $J$. By Proposition \ref{modal equivalence => bisimulation} and Corollaries \ref{elementary to ultrapower} and \ref{useful coro}  we obtain a sequence of implications:
\begin{tabbing}
$\mathfrak{M}'\models_s\Phi(u)[u:=a']$\hskip2mm\= $\Lra$\hskip2mm\= $\prod_U\mathfrak{M'}\models_s\Phi(u)[u:=f'_{a,U}]$\\
\>$\Lra$\>$\prod_U\mathfrak{M}\models_s\Phi(u)[u:=f_{a,U}]$\\
\>$\Lra$\> $\mathfrak{M}\models_s\Phi(u)[u:=a]$
\end{tabbing}
This establishes that $m^1_1(\Phi)\models_s\Phi$.
The argument for a formula $\Psi(v)\in\mathcal{L}^1$, with $v\in V_\partial$ is similar.
\end{proof}

\subsection{Van Benthem Characterization}
Fix a similarity type $\tau$. Let $\Lambda_\tau$ be the language of normal lattice expansions of type $\tau$, $L_\tau=(L_1,L_\partial)_\tau$ be the sorted modal language of type $\tau$ and $\mathcal{L}^1_{s,\tau}$ the sorted first-order language of the same type $\tau$. All the necessary work to lift van Benthem's characterization theorem to sorted modal logic has been presented and we state the result.

\begin{thm}\rm
\label{sorted vb}
Let $\Phi(u)\in\mathcal{L}^1_{s,\tau}$ be a formula in one free variable in the sorted first-order language $\mathcal{L}^1_{s,\tau}$, with $u\in V_1$. Then $\Phi$ is equivalent to the translation ST$_u(\alpha)$ of a modal formula $\alpha\in_1 L_\tau$ iff $\Phi$ is bisimulation invariant. Similarly for a formula $\Psi(v)$ with $v\in V_\partial$.
\end{thm}
\begin{proof}
If $\Phi$ is equivalent to the translation ST$_u(\alpha)$ of a modal formula \mbox{$\alpha\in_1 L_\tau$,} then $\Phi$ is bisimulation invariant by Proposition \ref{modal=>invariant}. For the converse, by Lemma \ref{inv bisim lemma} we obtain $m^1_1(\Phi)\models_s\Phi$.
By compactness for sorted FOL (Theorem \ref{sorted compactness}), let $\mu_1(\Phi)=\{\stx{u}{\alpha_1},\ldots,\stx{u}{\alpha_n}\} \subseteq m^1_1(\Phi)$ be a finite subset of $m^1_1(\Phi)$ such that $\mu_1(\Phi)\models_s\Phi$. Then $\models_s\Phi\leftrightarrow\bigwedge\mu_1(\Phi)$, hence $\models_s\Phi\leftrightarrow\stx{u}{\eta}$, where we set $\eta=\alpha_1\wedge\cdots\wedge\alpha_n$.
\end{proof}

It remains to adapt the result to the case of the logics of normal lattice expansions of similarity type $\tau$.

\begin{defn}\rm
$\Phi(u)$, with $u\in V_1$, is {\em stable} if and only if  it is  equivalent to the formula $\forall^\partial v\;\exists^1 z\;({\bf I}(u,v)\lra{\bf I}(z,v)\wedge\Phi(z))$.
\end{defn}

\begin{thm}[van Benthem Characterization]\rm
\label{van Benthem}
Fix a similarity type $\tau$.
Let $\Phi(u)\in\mathcal{L}^1_{s,\tau}$ be a formula with one free variable  in the sorted first-order language $\mathcal{L}^1_{s,\tau}$, with $u\in V_1$. Then $\Phi$ is equivalent to the translation ST$^\bullet_\pi(\varphi)$, for some permutation $\pi:\omega\ra\omega$, of a sentence in the language of lattice expansions of similarity type $\tau$ iff $\Phi$ is bisimulation invariant and stable.
\end{thm}
\begin{proof}
The claim of the theorem follows immediately by combining the characterization result for sorted modal logic of similarity type $\tau$ (Theorem \ref{sorted vb}) and Corollary \ref{modal characterization} (a consequence of Theorem \ref{properties of trans}).
\end{proof}

\section{Conclusions}
This article is part of a project of employing modal methods and lifting results proved for modal logic to the case of non-distributive propositional logics, i.e. the logics of normal lattice expansions of some similarity type $\tau$. An intermediate step in carrying out the proof has been the lifting of the van Benthem characterization result to the case of sorted modal logic, which is unproblematic, though burdened with the usual technicalities one needs to deal with when moving from unsorted to sorted domains. Its core idea, on which it relies heavily, is the modal representation of normal lattice expansions developed in \cite{duality2021,sdl-exp,pnsds,discres} and the possibility to provide fully abstract modal translations of the languages of logics for normal lattice expansions into sorted modal logic, an idea first explored in \cite{pll7,redm}.

\bibliographystyle{plain}

\begin{thebibliography}{10}

\bibitem{bell-slomson}
J.L. Bell and A.B. Slomson.
\newblock {\em Models and Ultraptoducts: An Introduction}.
\newblock North Holand, 1969.

\bibitem{birkhoff}
Garrett Birkhoff.
\newblock {\em Lattice theory}.
\newblock American Mathematical Society Colloquium Publications 25, American
  Mathematical Society, Providence, Rhode Island, third edition, 1979.
\newblock (corrected reprint of the 1967 third edition).

\bibitem{modlog}
Patrick Blackburn, Maarten de~R\"{i}jke, and Yde Venema.
\newblock {\em Modal Logic}, volume~53 of {\em Cambridge Tracts in Theoretical
  Computer Science}.
\newblock CUP, Cambridge, 2001.

\bibitem{chang-keisler}
C.~C. Chang and H.~J. Keisler.
\newblock {\em Model Theory}.
\newblock Studies in Logic and the Foundations of Mathematics. Elsevier, 1990.

\bibitem{mai-grishin}
Anna Chernilovskaya, Mai Gehrke, and Lorijn van Rooijen.
\newblock Generalised {K}ripke semantics for the {L}ambek-{G}rishin calculus.
\newblock {\em Logic Journal of the IGPL}, 20(6):1110--1132, 2012.

\bibitem{conradie-palmigiano}
Willem Conradie and Alessandra Palmigiano.
\newblock Algorithmic correspondence and canonicity for non-distributive
  logics.
\newblock {\em Ann. Pure Appl. Logic}, 170(9):923--974, 2019.

\bibitem{craig}
Andrew Craig, Maria~Joao Gouveia, and Miroslav Haviar.
\newblock {T}i{RS} graphs and {T}i{RS} frames: a new setting for duals of
  canonical extensions.
\newblock {\em Algerba Universalis}, 74(1-2), 2015.

\bibitem{dunn-gehrke}
Jon~Michael Dunn, Mai Gehrke, and Alessandra Palmigiano.
\newblock Canonical extensions and relational completeness of some
  substructural logics.
\newblock {\em Journal of Symbolic Logic}, 70:713--740, 2005.

\bibitem{vac-et-al}
Ivo D\"{u}ntsch, Ewa Orlowska, Anna~Maria Radzikowska, and Dimiter Vakarelov.
\newblock Relational representation theorems for some lattice-based structures.
\newblock {\em Journal of Relational Methods in Computer Science (JORMICS)},
  1:132--160, 2004.

\bibitem{enderton}
Herbert~B. Enderton.
\newblock {\em A {M}athematical {I}ntroduction to {L}ogic}.
\newblock Academic {P}ress, 1972.

\bibitem{ono-galatos}
N~Galatos, P.~Jipsen, T.~Kowalski, and H.~Ono.
\newblock {\em Residuated lattices: An algebraic glimpse at substructural
  logics}, volume 151 of {\em Studies in logic and the foundations of
  mathematics}.
\newblock Elsevier, 2007.

\bibitem{wille2}
Bernhard Ganter and Rudolph Wille.
\newblock {\em Formal Concept Analysis: Mathematical Foundations}.
\newblock Springer, 1999.

\bibitem{mai-gen}
Mai Gehrke.
\newblock Generalized {K}ripke frames.
\newblock {\em Studia Logica}, 84(2):241--275, 2006.

\bibitem{goldb}
Robert Goldblatt.
\newblock Semantic analysis of orthologic.
\newblock {\em Journal of Philosophical Logic}, 3:19--35, 1974.

\bibitem{goldblatt-ultra-2018}
Robert Goldblatt.
\newblock Canonical extensions and ultraproducts of polarities.
\newblock {\em Algebra universalis}, 79(4):80, Oct 2018.

\bibitem{goldblatt-definable2018}
Robert Goldblatt.
\newblock Definable operators on stable set lattices.
\newblock {\em Studia Logica}, 108(6):1263--1280, 2020.

\bibitem{goldblatt-morphisms2019}
Robert Goldblatt.
\newblock Morphisms and duality for polarities and lattices with operators.
\newblock {\em FLAP}, 7:1017--1070, 2020.

\bibitem{odigpl}
Chrysafis Hartonas.
\newblock Order-dual relational semantics for non-distributive propositional
  logics.
\newblock {\em Oxford Logic Journal of the IGPL}, 25(2):145--182, 2017.

\bibitem{sdl-exp}
Chrysafis Hartonas.
\newblock Stone {D}uality for {L}attice {E}xpansions.
\newblock {\em Oxford Logic Journal of the {IGPL}}, 26(5):475--504, 2018.

\bibitem{discres}
Chrysafis Hartonas.
\newblock Duality results for (co)residuated lattices.
\newblock {\em Logica Universalis}, 13(1):77--99, 2019.

\bibitem{pll7}
Chrysafis Hartonas.
\newblock Lattice logic as a fragment of (2-sorted) residuated modal logic.
\newblock {\em Journal of Applied Non-Classical Logics}, 29(2):152--170, 2019.

\bibitem{redm}
Chrysafis Hartonas.
\newblock Modal translation of substructural logics.
\newblock {\em Journal of Applied Non-Classical Logics}, 30(1):16--49, 2020.

\bibitem{duality2021}
Chrysafis Hartonas.
\newblock Duality for normal lattice expansions and sorted, residuated frames
  with relations,  2021.
\newblock http://arxiv.org/abs/2110.06924v1.

\bibitem{kata2z}
Chrysafis Hartonas.
\newblock Reconcilliation of approaches to the semantics of logics without
  distribution,  2021.
\newblock http://arxiv.org/abs/2109.11597.

\bibitem{sdl}
Chrysafis Hartonas and J.~Michael Dunn.
\newblock Stone duality for lattices.
\newblock {\em Algebra Universalis}, 37:391--401, 1997.

\bibitem{pnsds}
Chrysafis Hartonas and Ewa Or{\l}owska.
\newblock Representation of lattices with modal operators in two-sorted frames.
\newblock {\em Fundam. Inform.}, 166(1):29--56, 2019.

\bibitem{hartung}
Gerd Hartung.
\newblock A topological representation for lattices.
\newblock {\em Algebra Universalis}, 29:273--299, 1992.

\bibitem{jt1}
Bjarni J\'{o}nsson and Alfred Tarski.
\newblock Boolean algebras with operators {I}.
\newblock {\em American Journal of Mathematics}, 73:891--939, 1951.

\bibitem{jt2}
Bjarni J\'{o}nsson and Alfred Tarski.
\newblock Boolean algebras with operators {II}.
\newblock {\em American Journal of Mathematics}, 74:8127--162, 1952.

\bibitem{ono3}
Hiroakira Ono.
\newblock Algebraic aspects of logics without structural rules.
\newblock {\em AMS, Contemporary Mathematics}, 131:601--621, 1992.

\bibitem{vak-ewa}
E.~Orlowska and D.~Vakarelov.
\newblock Lattice-based modal algebras and modal logics.
\newblock In Petr H\'{a}jek, Luis Vald\'{e}s-Villanueva, and Dag
  Westerst{\aa}hl, editors, {\em Logic, Methodology and Philosophy of Science,
  Proceedings of the Twelfth International Congress (7–13 August 2003,
  Oviedo, Spain)}, pages 147--170. King’s College Publications, 2005.

\bibitem{plo}
Miroslav Plo\v{s}\v{c}ica.
\newblock A natural representation of bounded lattices.
\newblock {\em Tatra Mountains Math. Publ.}, 5:75--88, 1995.

\bibitem{suzuki8}
Tomoyuki Suzuki.
\newblock Bi-approximation semantics for substructural logic at work.
\newblock In {\em Advances in Modal Logic vol 8}, pages 411--433, 2010.

\bibitem{suzuki-shalqvist}
Tomoyuki Suzuki.
\newblock A {S}ahlqvist theorem for substructural logic.
\newblock {\em Rev. Symb. Log.}, 6(2):229--253, 2013.

\bibitem{Suzuki-polarity-frames}
Tomoyuki Suzuki.
\newblock On polarity frames: Applications to substructural and lattice-based
  logics.
\newblock In Rajeev Gor\'e, Barteld Kooi, and Agi Kurucz, editors, {\em
  Advances in Modal Logic, Volume 10}, pages 533--552. CSLI Publications, 2014.

\bibitem{urq}
Alasdair Urquhart.
\newblock A topological representation of lattices.
\newblock {\em Algebra Universalis}, 8:45--58, 1978.

\end{thebibliography}

\end{document}